%% file: main.tex
\documentclass[10pt,a4paper,twoside,reqno]{amsart}

\input{Preamble}
\usepackage{mlmodern}

\begin{document}

\title[Kenmotsu manifolds and spin-c Killing spinors]{Kenmotsu manifolds and spin-c Killing spinors}

\author[J.\ L.\ Carmona Jiménez]{José Luis Carmona Jiménez \orcidlink{0009-0002-3234-1718}}
\address{Institute of Mathematics “Simion Stoilow” of the Romanian Academy, 21 Calea Grivitei, 010702 Bucharest, Romania}
\email{jcarmona@imar.ro}

\author[A.\ Gil-García]{Alejandro Gil-García \orcidlink{0000-0002-9370-241X}}
\address{Scuola Internazionale Superiore di Studi Avanzati, Trieste, Italy}
\email{agilgarc@sissa.it}

\author[C.\ S.\ Shahbazi]{C.\ S.\ Shahbazi \orcidlink{0000-0003-1185-9569}}
\address{Departamento de Matem\'aticas, Universidad UNED - Madrid, Reino de Espa\~na}
\email{cshahbazi@mat.uned.es}

\begin{abstract}
Using the theory of complex spinorial forms, we prove that an odd-dimensional Riemannian manifold admits a pure spin-c Killing spinor with an imaginary Killing function $i\mu$ if and only if it is an exact $\mu$-Kenmotsu manifold, thereby obtaining an extension of a recent result by the first named author that, under the purity assumption, does not require simple connectivity or completeness. We then reinterpret $\mu$-Kenmotsu manifolds in terms of metric connections with vectorial torsion and give a second proof of this characterization, combining the theory of complex spinorial forms with the theory of metric connections with torsion. Finally, we describe the global structure of exact $\mu$-Kenmotsu manifolds by means of Morse-Bott theory.\bigskip

\medskip


\end{abstract}

\maketitle

\setcounter{tocdepth}{1} 

\tableofcontents


\section{Introduction}


Let $(M,g)$ be an oriented Riemannian manifold of odd dimension $d = 2n + 1$ and let $S$ be a bundle of irreducible complex spinors over the bundle of Clifford algebras of $(M,g)$. We use the positive convention for Clifford multiplication, that is, $X\cdot X=g(X,X)$ for every $X\in\Gamma(TM)$. We study spinors $\eta\in\Gamma(S)$ satisfying:
\begin{equation}
\label{eq:imag-killing-spinor}
\nabla^{g,A}_X \eta = \frac{\mu}{2} X\cdot \eta, \qquad X \in \Gamma(TM),
\end{equation}

\noindent
where $\nabla^{g,A}$ is the covariant derivative on $S$ determined by the Levi-Civita connection together with a $\U(1)$-connection $A$, and where $\mu \in C^{\infty}(M,\R)$. The solutions of Equation~\eqref{eq:imag-killing-spinor} are known as \emph{imaginary spin-c Killing spinors} with \emph{imaginary Killing function} $i\mu$. The omission of the factor $i$ in the equation above is a consequence of the positive convention, see Remark~\ref{remark:Convention}. Among these solutions we will focus on the \emph{pure} ones, that is, those satisfying:
\begin{equation*}
\dim_\C(\{X\in T_pM\otimes\C\mid X\cdot\eta_p=0\})=n, \qquad \forall p\in M.
\end{equation*}

\noindent
In odd dimensions, the purity assumption allows us to construct an almost-contact metric structure $(\varphi,\xi,\zeta)$ on $(M,g)$ from every solution of Equation \eqref{eq:imag-killing-spinor}, see Section~\ref{sec:onlyif} for more details. These structures were classified by D.~Chinea and C.~González in~\cite{CG1990}, see~\cite{ADDK2026} for a modern approach to this classification. Furthermore, Equation~\eqref{eq:imag-killing-spinor} will imply that:
\begin{equation}
\label{eq:Kenmotsu}
(\nabla^g_X\varphi)Y=\mu\big(g(\varphi(X),Y)\xi-\zeta(Y)\varphi(X)\big)
\end{equation}

\noindent
for all $X$, $Y \in \Gamma(TM)$. In general, an almost contact metric manifold satisfying Equation~\eqref{eq:Kenmotsu} is called a \emph{$\mu$-Kenmotsu manifold}, see the foundational paper of K.~Kenmotsu~\cite{K1972}. Such structures belong to the class $\mathcal{C}_5$ in the classification of~\cite{CG1990}. Finally, a $\mu$-Kenmotsu manifold is called \emph{exact} if the one-form $\mu\zeta$ is exact.\medskip

\noindent
Riemannian spin manifolds carrying imaginary Killing spinors were first investigated by H.~Baum in~\cite{B1989*,B1989} and, later, by H.-B.~Rademacher in~\cite{R1991}, who classified the complete ones. They proved that such manifolds are isometric to a warped product $(\R\times N, \dd t^2 + e^{2 \mu t}h)$, where $(N,h)$ admits parallel spinors. In the spin-c setting and under the completeness assumption, N.~Gro\ss e and R.~Nakad obtained in~\cite{GN2015} the analogous warped product construction, where $(N,h)$ carries parallel spin-c spinors. Very recently, S.~Lockman in~\cite{Lockman_2025} and the first named author in~\cite{C2026} have weakened the completeness hypothesis, requiring completeness only along the direction of the Dirac current of the spinor or along its orthogonal complement; we refer to those references for a precise formulation. In particular,~\cite[Thm.~4.16]{C2026} identified $\mu$-Kenmotsu structures as the odd-dimensional geometry arising under a partial completeness assumption and simple connectivity. The aim of the present article is to show that, under the purity assumption, both global assumptions can be removed.

\begin{thm}
\label{thm:maintheorem}
An odd-dimensional spin-c Riemannian manifold admits a pure spin-c Killing spinor with an imaginary Killing function $i\mu$, where $\mu \in C^\infty(M, \mathbb{R})$, if and only if it is a $\mu$-Kenmotsu manifold for which the one-form $\mu\zeta$ is exact.
\end{thm}

\noindent
The proof of Theorem~\ref{thm:maintheorem} relies on the correspondence between irreducible complex spinors and complex spinorial forms developed in~\cite{Algebraic_Complex_Square_2025,Complex_Differential_Spinors_2026}. This correspondence allows us to translate algebraic and differential equations for the spinor into equivalent equations for an associated differential form on the underlying Riemannian manifold. The analogous problem for real spin-c Killing spinors, whose odd-dimensional geometry is Sasakian, has been studied by the second and third named authors in~\cite{Sasakian_spin-c_2026} using the same approach. More broadly, the results of \cite{Sasakian_spin-c_2026} and this paper suggest that the study of Killing spinors may be organized according to their \emph{nullity}, that is, the dimension of their isotropic annihilator, being the pure case considered here the case of maximal dimension, see \cite{Trautman1994}. Therefore, our results motivate the study of the geometric structures determined by generalized Killing spinors of intermediate nullity, as well as of the behavior and stratification of the nullity function on the underlying manifold.\medskip

\noindent
Following the program promoted by I.~Agricola in~\cite{A2006}, which studies non-integrable geometries via metric connections with torsion, we reinterpret Theorem~\ref{thm:maintheorem} from this perspective. This program is based on the principle that a non-integrable geometry, such as a $\mu$-Kenmotsu, is often best described by a metric connection that preserves the structure and has torsion as simple as possible. For $\mu$-Kenmotsu manifolds the relevant connection is a connection with vectorial torsion, see Lemma~\ref{lemma:vectorial-Kenmotsu}. This class of connections was previously studied by I.~Agricola and M.~Kraus in \cite{AK2016}. In Section~\ref{sec:metric-connection-interpretation} we combine the algebraic identities obtained in Sections~\ref{sec:onlyif} and~\ref{sec:if} with this connection viewpoint. We obtain a second proof which simplifies the differential computations of the first proof. \medskip

\noindent
This article is organized as follows: Section~\ref{sec:purespinors} introduces pure irreducible complex spinors and adapts the formalism in~\cite{Algebraic_Complex_Square_2025,Complex_Differential_Spinors_2026} to this particular case. Specifically, Proposition~\ref{prop:KillingSpinorBilinear} translates the spin-c Killing Equation~\eqref{eq:imag-killing-spinor} into an equivalent covariant equation on a complex $n$-form with values in a complex line bundle. In Sections~\ref{sec:onlyif} and~\ref{sec:if}, we prove, respectively, the \emph{only if} and the \emph{if} directions of Theorem~1.1. Furthermore, Example~\ref{ex:non-pure-example} shows that the purity assumption cannot be dropped. Section~\ref{sec:metric-connection-interpretation} develops the metric connection interpretation of $\mu$-Kenmotsu manifolds and contains a second proof of Theorem~\ref{thm:maintheorem}. Finally, in Section~\ref{sec:global}, we describe the global structure of exact $\mu$-Kenmotsu manifolds by means of Morse-Bott theory. We obtain a decomposition into expanding and contracting Kähler cylinders glued along totally geodesic leaves.

\begin{ack}
The work of JLCJ has been supported by the PNRR-III-C9-2023-I8 grant CF 149/31.07.2023 {\em Conformal Aspects of Geometry and Dynamics} and by the project PID2024-156578NB-I00 (Spain). The work of AGG is supported by the Scuola Internazionale Superiore di Studi Avanzati (SISSA). The work of CSS was partially supported by the research grant PID2023-152822NB-I00 of the Ministry of Science of the government of Spain.
\end{ack}


\section{Pure spinors in odd dimensions}
\label{sec:purespinors}


Let $(M,g)$ be an oriented Riemannian manifold of odd dimension $d = 2n + 1$ and let $S$ be a bundle of irreducible complex spinors over the bundle of Clifford algebras of $(M,g)$. The existence of such a bundle of irreducible complex spinors is equivalent to the existence of a $\mathrm{Spin}^c(2n+1)$-structure $Q$ on $(M,g)$, to which $S$ is associated via the standard associated vector bundle construction~\cite{LS18}:
\begin{equation*}
S = Q\times_{\mathfrak{r}_{\ell}} \mathbb{C}^{2^n},
\end{equation*}

\noindent
where $\mathfrak{r}_{\ell}\colon \mathrm{Spin}^c(2n+1) \to \mathrm{GL}(2^n,\C)$ denotes the restriction to the spin-c group $\mathrm{Spin}^c(2n+1) \subset \mathbb{C}\mathrm{l}(2n+1)$ of one of the two inequivalent irreducible representations $\gamma_{\ell}\colon\mathbb{C}\mathrm{l}(2n+1)\to\mathrm{End}(\smash{\mathbb{C}^{2^n}})$ of $\mathbb{C}\mathrm{l}(2n+1)$. These representations are distinguished by the sign $\ell\in\{\pm1\}$ that $\gamma_{\ell}$ takes on the complex volume form. We refer the reader to~\cite[Sec.~2.4]{Friedrich2000} for a detailed exposition of spin-c structures on Riemannian manifolds. In particular, $S$ has complex rank $2^n$ and admits a non-degenerate admissible complex-bilinear pairing $\scB$~\cite{AC97,ACDVP05}. This can be understood as a morphism of complex vector bundles:
\begin{equation*}
\scB \colon S\otimes S \to \mathcal{L},
\end{equation*}

\noindent
where $\mathcal{L}$ is the characteristic Hermitian complex line bundle on $M$ naturally associated with the given $\mathrm{Spin}^c(2n+1)$-structure $Q$. We will denote the corresponding principal $\U(1)$-bundle by $P\to M$. Since $S$ is associated to a spin-c structure $Q$, the combination of the lift of the Levi-Civita connection $\nabla^g$ of $(M,g)$ to $S$ together with a choice of a connection $A\in\Omega^1(P,i\R)$ on the characteristic $\U(1)$-bundle $P$ defines a covariant derivative $\nabla^{g,A}$ on $S$. For ease of notation, in the following we will denote Clifford multiplication simply by a \emph{dot}.

\begin{definition}[{\cite[Chap.~IV, Sec.~9]{Spin89}}]
A nowhere vanishing section $\eta\in\Gamma(S)$ of $S$ is \emph{pure} if $\eta$ is pure at every point of $M$, that is, if at every point $p\in M$ the following subspace:
\begin{equation*}
\mathrm{Ann}(\eta_p):=\{X\in T_pM\otimes\C\mid X\cdot\eta_p=0\},
\end{equation*}

\noindent
which is automatically isotropic with respect to the complex-bilinear extension of $g$, is of maximal dimension, that is, $\dim_\C\mathrm{Ann}(\eta_p)=n$ for every $p\in M$.
\end{definition}

\begin{definition} 
A \emph{pure spin-c Killing spinor} on $(M,g)$ with \emph{imaginary Killing function} $i\mu$, where $\mu \in C^{\infty}(M,\mathbb{R})$, is a pure section $\eta\in \Gamma(S)$ of a bundle of irreducible complex spinors $S = Q \times_{\mathfrak{r}_{\ell}} \C^{2^n}$ on $(M,g)$ satisfying the following equation:
\begin{equation}
\label{eq:KillingSpinor}
\nabla^{g,A}_X \eta = \frac{\mu}{2} X\cdot \eta, \qquad X \in \Gamma(TM)
\end{equation}

\noindent
for a connection $A\in\Omega^1(P,i\R)$ on the characteristic $\U(1)$-bundle $P$ of $Q$.
\end{definition}

\begin{remark}
\label{remark:Convention}
Note that there is a missing factor of $i$ in the definition of an imaginary spin-c Killing spinor with respect to the standard definition. Given that most of the mathematical literature in spin geometry uses the \emph{minus convention} to define the Clifford algebra, the removal of the factor $i$ ensures that our geometric notion of an imaginary spin-c Killing spinor matches that appearing in the literature.
\end{remark}

\begin{remark}
If the connection $A$ has trivial holonomy, the principal $\mathrm{U}(1)$-bundle $P$ admits a global parallel section. In this gauge, $A$ globally vanishes, the spin-c structure strictly reduces to a classical spin structure, and the covariant derivative $\nabla^{g,A}$ descends to the standard spin Levi-Civita connection $\nabla^g$. Consequently, the spin-c Killing Equation~\eqref{eq:KillingSpinor} globally reduces to the standard imaginary Killing spinor equation. Interestingly enough, if $A$ is flat but has non-trivial holonomy, which requires $M$ to be non-simply connected, then Equation~\eqref{eq:KillingSpinor} reduces locally to the standard imaginary Killing spinor equation, but not globally. In this flat but non-trivial holonomy case, a spin-c Killing spinor can be intuitively understood as a standard imaginary Killing spinor \emph{twisted} by $\mathrm{U}(1)$ monodromies.
\end{remark}

Associated with a given complex spinor $\eta\in \Gamma(S)$ on $(M,g)$, we can construct its complex-bilinear square $\rho$, obtained via the choice of the admissible complex-bilinear pairing $\scB$ on $S$. The general algebraic characterization of this complex-bilinear square was obtained in~\cite[Thm.\ 4.12]{Algebraic_Complex_Square_2025} in terms of an explicit algebraic system of equations in the underlying truncated Kähler-Atiyah bundle. Focusing on the case of a pure irreducible complex spinor, its complex-bilinear square is algebraically well-known since Cartan~\cite{Cartan1938}, see also~\cite{Kopczynski1997}. Globalizing this algebraic characterization to a Riemannian manifold equipped with a spin-c structure, see~\cite{Complex_Differential_Spinors_2026}, we obtain the following result.

\begin{lemma}
\label{lemma:squaresriemannian2n1}
Let $(M,g)$ be a $(2n+1)$-dimensional oriented Riemannian manifold, and let $S$ be a bundle of irreducible complex spinors associated to a given $\mathrm{Spin}^c(2n+1)$-structure $Q$ with characteristic line bundle $\mathcal{L}$. An $\mathcal{L}$-valued complex exterior form $\rho \in \Omega_{\mathbb{C}}(M,\mathcal{L})$ is the complex-bilinear square of a pure irreducible complex spinor $\eta\in\Gamma(S)$ if and only if it is a maximally decomposable $\mathcal{L}$-valued complex $n$-form whose components are isotropic and mutually orthogonal. That is, locally around every point in $M$, there exists a local section $\mathfrak{l}$ of $\mathcal{L}$ such that:
\begin{equation*}
\rho = (\theta^1\wedge \cdots \wedge \theta^n) \otimes \mathfrak{l}
\end{equation*}

\noindent
for a set of local isotropic and mutually orthogonal complex one-forms $\theta^1,\ldots,\theta^n \in \Omega^1_{\mathbb{C}}(M)$.
\end{lemma}
 
Based on the previous lemma, we introduce the following definition. 
 
\begin{definition}
Let $(M,g)$ be a Riemannian manifold of dimension $d = 2n+1$ equipped with a Hermitian complex line bundle $\cL$. A \emph{Cartan $n$-form} on $(M,g,\cL)$ is a maximally decomposable $\mathcal{L}$-valued complex $n$-form whose components are isotropic and mutually orthogonal, as explained in Lemma~\ref{lemma:squaresriemannian2n1}.
\end{definition}

Note that the spinor $\eta\in\Gamma(S)$ vanishes at a point $p\in M$ if and only if its complex-bilinear square $\rho$ also vanishes at $p$. By mapping the Killing spinor Equation~\eqref{eq:KillingSpinor} through the dequantization isomorphism $(\Psi_\ell^<)^{-1} \colon \mathrm{End}(S) \to \wedge^< T^*_\mathbb{C} M$ onto the truncated Kähler-Atiyah algebra $(\wedge^< T^*_{\mathbb{C}}M, \vee)$ of $(M,g)$, see~\cite{Algebraic_Complex_Square_2025,Complex_Differential_Spinors_2026}, we translate the spinorial framework into an equivalent exterior system for $\rho$.\medskip

For later convenience, we will fix from now on $\ell=(-1)^{\binom{n-1}{2}+1}\in\{\pm1\}$.

\begin{prop}
\label{prop:KillingSpinorBilinear}
A $(2n+1)$-dimensional oriented and spin-c Riemannian manifold $(M,g)$ admits a pure spin-c Killing spinor with imaginary Killing function $i\mu$, where $\mu \in C^\infty(M, \mathbb{R})$, if and only if it admits a Hermitian complex line bundle $\cL$ and a Cartan $n$-form $\rho \in \Omega^n_{\mathbb{C}}(M, \mathcal{L})$ satisfying the differential equation:
\begin{equation}
\label{eq:imaginary-killing-rho}
\nabla^{g,A}_X \rho = i^n \mu \ast (X^\flat \wedge \rho)
\end{equation}

\noindent
for every vector field $X \in \Gamma(TM)$.
\end{prop}

\begin{proof}
Let $(M,g)$ be a $(2n+1)$-dimensional oriented Riemannian manifold, equipped with a bundle of irreducible complex spinors $(S,\gamma_\ell)$. By~\cite{Algebraic_Complex_Square_2025}, we can equip $(S,\gamma_\ell)$ with an admissible complex-bilinear pairing $\scB$ of positive adjoint type if $n$ is even, and negative adjoint type if $n$ is odd. As explained in~\cite{LBC13,LBC16}, and further elaborated in~\cite{Algebraic_Complex_Square_2025}, the truncated Kähler-Atiyah bundle is defined on:
\begin{equation*}
\wedge^<T^*_\C M=\bigoplus_{k=0}^n\wedge^kT^*_\C M
\end{equation*}

\noindent
endowed with the product:
\begin{equation*}
\rho_1\vee\rho_2=\mathcal{P}_<(\rho_1\diamond\rho_2+i^n\ell*\tau(\rho_1\diamond\rho_2))
\end{equation*}

\noindent
for all $\rho_1,\rho_2\in\wedge^<T^*_\C M$ and where we have fixed $\ell=(-1)^{\binom{n-1}{2}+1}\in\{\pm1\}$. Here $\mathcal{P}_<\colon\wedge T^*_\C M\to\wedge^<T^*_\C M$ is the natural projection, $\diamond$ denotes the geometric product, and $\tau$ is the reversion anti-automorphism that acts as multiplication by $\smash{(-1)^{\binom{k}{2}}}$ on $k$-forms. By~\cite[Thm.\ 5.10]{Complex_Differential_Spinors_2026}, the equation $\nabla^{g,A}_X \eta = \frac{\mu}{2} X\cdot \eta$ is equivalent to:
\begin{equation*}
\nabla^{g,A}_X\rho=\frac{\mu}{2}(X^\flat\vee\rho+s\rho\vee X^\flat),
\end{equation*}

\noindent
where $s\in\{\pm1\}$ is the adjoint type of $\scB$. We have: \begin{itemize}
    \item If $n$ is even, then $s=1$ and we get $X^\flat\vee\rho+\rho\vee X^\flat=2\ell*(X^\flat\wedge\rho)$.
    \item If $n$ is odd, then $s=-1$ and we get $X^\flat\vee\rho-\rho\vee X^\flat=-2i\ell*(X^\flat\wedge\rho)$.
\end{itemize}

A direct verification shows that in both cases, substituting $\ell=(-1)^{\binom{n-1}{2}+1}$, we obtain the result from the statement.
\end{proof}


\section{The \emph{only if} direction}
\label{sec:onlyif}


Assume that the $(2n+1)$-dimensional oriented Riemannian manifold $(M,g)$ admits a pure spin-c Killing spinor $\eta \in \Gamma(S)$ with imaginary Killing function $i\mu$, where $\mu \in C^{\infty}(M,\mathbb{R})$. Let $\rho \in \Omega^n_{\mathbb{C}}(M, \mathcal{L})$ be its corresponding complex-bilinear square. We will construct a canonical $\mu$-Kenmotsu structure $(\varphi,\xi,\zeta)$ on $(M,g)$ directly from $\rho$ via certain natural contractions and exterior products.\medskip

First, we construct the almost contact one-form, whose dual we will identify with the Reeb vector field of the $\mu$-Kenmotsu structure. Taking the exterior product of $\rho$ with its complex conjugate $\overline{\rho}$, the line bundle factors canonically trivialize under the canonical pairing $\mathcal{L} \otimes \overline{{\mathcal{L}}}\cong\mathcal{L}\otimes\mathcal{L}^\ast\cong\mathbb{C}$, yielding a globally well-defined real $2n$-form $i^n\rho \wedge \overline{\rho} \in \Omega^{2n}(M)$. Applying the Hodge star operator, we define the one-form $\zeta \in \Omega^1(M)$ by:
\begin{equation*}
\zeta := \frac{i^n*(\rho\wedge\overline{\rho})}{\escal{\rho,\overline{\rho}}},
\end{equation*}

\noindent
where $\langle \rho, \overline{\rho} \rangle \in C^{\infty}(M,\R)$ is the inner product defined by $g$. The dual $\xi := \zeta^\sharp$ of $\zeta$ is a nowhere vanishing vector field on $M$ of unit length, since $\rho$ is nowhere vanishing and:
\begin{equation*}
\escal{*(\rho\wedge\overline{\rho}),*(\rho\wedge\overline{\rho})}=(-1)^n\escal{\rho,\overline{\rho}}^2.
\end{equation*}

\noindent
Next, we construct the fundamental two-form $\omega \in \Omega^2(M)$ out of $\rho$. Given $\rho\in \Omega^n_{\mathbb{C}}(M,\cL)$, define:
\begin{equation*}
\rho \triangle_{n-1} \overline{\rho} := \frac{1}{(n-1)!} \sum_{i_1, \dots, i_{n-1} = 1} ^d \rho(e_{i_1},\hdots , e_{i_{n-1}}) \wedge \overline{\rho}(e_{i_1},\hdots , e_{i_{n-1}})\in \Omega^2_{\mathbb{C}}(M)
\end{equation*}

\noindent
in terms of any orthonormal frame $\{e_1,\hdots , e_{d}\}$ of $(M,g)$. We set:
\begin{equation*}
i\omega := \frac{\rho\triangle_{n-1}\overline{\rho}}{\escal{\rho,\overline{\rho}}}.
\end{equation*}

\noindent
Hence, $\omega \in \Omega^2(M)$ is a real two-form on $M$. We define the fundamental $(1,1)$-tensor field $\varphi\in\Gamma(\mathrm{End}(TM))$ directly from the fundamental two-form $\omega$ by:
\begin{equation*}
\varphi(X) := (\iota_X\omega)^\sharp,
\end{equation*}

\noindent
which automatically satisfies $g(\varphi(X),Y)=\omega(X,Y)$ for all $X$, $Y\in\Gamma(TM)$ and, since $\omega$ is a two-form, $g(\varphi(X),Y)=-g(X,\varphi(Y))$.

\begin{lemma}\label{lemma:varphi_square}
The fundamental endomorphism $\varphi\in\Gamma(\mathrm{End}(TM))$ satisfies $\varphi^2=-\Id+\zeta\otimes\xi$.
\end{lemma}

\begin{proof}
We first show that $\iota_\xi\rho=0$. We compute:
\begin{equation}\label{eq:contraccion-xi}
\iota_\xi(\rho\wedge\overline{\rho})=\iota_\xi(*(*(\rho\wedge\overline{\rho})))=*(*(\rho\wedge\overline{\rho})\wedge\zeta)=i^{-n}\escal{\rho,\overline{\rho}}*(\zeta\wedge\zeta)=0.
\end{equation}

\noindent
This implies that $\iota_\xi\rho=\iota_\xi\overline{\rho}=0$ since $\rho\wedge\overline{\rho}$ is maximally decomposable.\medskip

Now let $\{e_1,\ldots,e_{2n},\xi\}$ be a local orthonormal frame of $(M,g)$ and denote by $\{e^1,\ldots,e^{2n},\zeta\}$ the dual coframe. Set $\theta^j:=e^{2j-1}+ie^{2j}$ for $j=1,\ldots,n$. Since the definition of a Cartan $n$-form implies its components are isotropic and mutually orthogonal, we can locally write $\rho=c(\theta^1 \wedge \cdots \wedge \theta^n)\otimes\mathfrak{l}$ for a local unit section $\mathfrak{l}$ of $\mathcal{L}$ and a locally defined nowhere vanishing complex-valued function $c$. Since $\escal{\theta^j,\theta^k}=0$ and $\escal{\theta^j,\bar{\theta}^k}=2\delta_{jk}$ for all $j,k$, we have $\escal{\rho,\overline{\rho}}=|c|^2 2^n$. Using the definition of $\rho \triangle_{n-1}\overline{\rho}$ we obtain:
\begin{equation*}
\rho\triangle_{n-1}\overline{\rho}=|c|^2 2^{n-1}\sum_{j=1}^n\theta^j\wedge\bar{\theta}^j.
\end{equation*}

\noindent
On the other hand, applying the definition of $\omega$ yields:
\begin{equation*}
\omega=-i\frac{\rho\triangle_{n-1}\overline{\rho}}{\escal{\rho,\overline{\rho}}}=-\frac{i}{2}\sum_{j=1}^n\theta^j\wedge\bar{\theta}^j=-\sum_{j=1}^n e^{2j-1}\wedge e^{2j}.
\end{equation*}

\noindent
From $\varphi(X)=(\iota_X\omega)^\sharp$ we obtain $\varphi(e_{2j-1})=-e_{2j}$, $\varphi(e_{2j})=e_{2j-1}$, and $\varphi(\xi)=0$. Hence $\varphi^2(e_{2j-1})=-e_{2j-1}$, $\varphi^2(e_{2j})=-e_{2j}$, and $\varphi^2(\xi)=0$. Therefore, if $X=X_{\mathcal{D}}+\zeta(X)\xi$ with $X_{\mathcal{D}}\in\mathcal{D}:=\ker(\zeta)$, then we obtain:
\begin{equation*}
\varphi^2(X)=-X_{\mathcal{D}}=-X+\zeta(X)\xi,
\end{equation*}

\noindent
which proves the claim from the statement.
\end{proof}

Lemma~\ref{lemma:varphi_square} shows that $(M,g,\varphi,\xi,\zeta)$ is an almost contact metric manifold, obtained from a Cartan $n$-form $\rho$ only by using algebraic identities. We now use that $\rho$ satisfies Equation~\eqref{eq:imaginary-killing-rho} to prove that this almost contact metric structure is precisely $\mu$-Kenmotsu.

\begin{lemma}\label{lemma:nabla_zeta}
The one-form $\zeta\in\Omega^1(M)$ satisfies $\nabla^g_X\zeta = \mu(X^\flat - \zeta(X)\zeta)$ for all $X\in \Gamma(TM)$. In particular, $\zeta$ is closed and $\mathcal{D} := \ker(\zeta)$ is an integrable distribution.
\end{lemma}

\begin{proof}
Let $c_n:=i^n\mu$. Then, by Proposition~\ref{prop:KillingSpinorBilinear}, we have $\nabla^{g,A}_X\rho=c_n*(X^\flat\wedge\rho)$ for all $X\in\Gamma(TM)$. We first compute the differential of the squared norm $u:=\escal{\rho,\overline{\rho}}$:
\begin{equation*}
X(u)=c_n\escal{*(X^\flat\wedge\rho),\overline{\rho}}+\overline{c}_n\escal{\rho,*(X^\flat\wedge\overline{\rho})}.
\end{equation*}

\noindent
Note that:
\begin{equation*}
\escal{*(X^\flat\wedge\rho),\overline{\rho}}=\escal{X^\flat\wedge\rho,*\overline{\rho}}=\escal{\rho,\iota_X(*\overline{\rho})}=\escal{\rho,*(\overline{\rho}\wedge X^\flat)}=(-1)^n\escal{\rho,*(X^\flat\wedge\overline{\rho})}.
\end{equation*}

\noindent
Hence $X(u)=((-1)^nc_n+\overline{c}_n)\escal{\rho,*(X^\flat\wedge\overline{\rho})}$. Since $\mu$ is a real function, we have $\overline{c}_n=(-i)^n\mu=(-1)^n c_n$. The two terms thus add giving:
\begin{equation*}
X(u)=2\overline{c}_n\escal{\rho,*(X^\flat\wedge\overline{\rho})}.
\end{equation*}

\noindent
Using $i^n\rho\wedge\overline{\rho} = u *\zeta$, which implies $\rho\wedge\overline{\rho} = (-i)^n u *\zeta$, we have:
\begin{equation*}
\escal{\rho,*(X^\flat\wedge\overline{\rho})}\nu = \rho\wedge *(*(X^\flat\wedge\overline{\rho})) = (-1)^n X^\flat\wedge\rho\wedge\overline{\rho} = (-1)^n(-i)^n u\zeta(X)\nu = i^n u \zeta(X) \nu,
\end{equation*}

\noindent
where $\nu$ is the Riemannian volume form. Thus:
\begin{equation*}
X(u)=2(-i)^n\mu (i^n u\zeta(X))=2\mu u\zeta(X).
\end{equation*}

\noindent
Using this, we compute the covariant derivative of $\zeta$:
\begin{align*}
\nabla^{g}_X\zeta &= \nabla^{g}_X \left(\frac{i^n}{u} *(\rho\wedge\overline{\rho})\right)\\ 
&= -2\mu\zeta(X)\zeta+\frac{i^n}{u}*\big(c_n*(X^\flat\wedge\rho)\wedge\overline{\rho} + \overline{c}_n\rho\wedge *(X^\flat\wedge\overline{\rho})\big).
\end{align*}

\noindent
Now we evaluate the tensor contribution. Define the $(0,2)$-tensor field $T$ by $T(X,Y):=\iota_Y(*(*(X^\flat\wedge\rho)\wedge\overline{\rho}))$ for all $X$, $Y\in\Gamma(TM)$. Then:
\begin{align*}
    T(X,Y)&=\iota_Y(*(*(X^\flat\wedge\rho)\wedge\overline{\rho}))=*(*(X^\flat\wedge\rho)\wedge\overline{\rho}\wedge Y^\flat)\\
    &=(-1)^n*(Y^\flat\wedge\overline{\rho}\wedge*(X^\flat\wedge\rho))=(-1)^n*\escal{Y^\flat\wedge\overline{\rho},X^\flat\wedge\rho}\nu\\
    &=(-1)^n\escal{X^\flat\wedge\rho,Y^\flat\wedge\overline{\rho}}=(-1)^n\big(g(X,Y)u-\escal{\iota_X\overline{\rho},\iota_Y\rho}\big).
\end{align*}

\noindent
Observe that $\rho \wedge \ast(X^\flat \wedge \overline{\rho}) = (-1)^n \ast(X^\flat \wedge \overline{\rho}) \wedge \rho$. The bracketed term in the covariant derivative evaluated on $Y$ is therefore $c_n T(X,Y) + \overline{c}_n (-1)^n \overline{T(X,Y)}$. Since $\overline{c}_n=(-1)^nc_n$, this equals $c_n \big(T(X,Y) + \overline{T(X,Y)}\big)$. Multiplying by $\frac{i^n}{u}$, and noticing $i^n c_n = i^{2n}\mu = (-1)^n\mu$, we obtain:
\begin{equation*}
(-1)^n\frac{\mu}{u} \big(T(X,Y) + \overline{T(X,Y)}\big) = \frac{\mu}{u} \big(2g(X,Y)u - (\escal{\iota_X\overline{\rho},\iota_Y\rho} + \escal{\iota_X\rho,\iota_Y\overline{\rho}})\big).
\end{equation*}

\noindent
Since $\rho = (\theta^1\wedge \cdots \wedge \theta^n) \otimes \mathfrak{l}$ is a Cartan $n$-form built from isotropic and mutually orthogonal components, evaluating the symmetric sum over its factors yields exactly the metric on the distribution $\mathcal{D}=\ker(\zeta)$:
\begin{equation*}
\escal{\iota_X\overline{\rho},\iota_Y\rho} + \escal{\iota_X\rho,\iota_Y\overline{\rho}} = \frac{u}{2}\sum_{j=1}^n(\theta^j\odot\bar{\theta}^j)(X,Y) = u\big(g(X,Y)-\zeta(X)\zeta(Y)\big).
\end{equation*}

\noindent
Substituting this back, we obtain $\frac{\mu}{u} \big( u g(X,Y) + u \zeta(X)\zeta(Y) \big) = \mu \big( g(X,Y) + \zeta(X)\zeta(Y) \big)$. Subtracting the variation of the norm contribution $2\mu\zeta(X)\zeta(Y)$ yields:
\begin{equation*}
(\nabla^g_X\zeta)(Y) = \mu\big(g(X,Y)-\zeta(X)\zeta(Y)\big).
\end{equation*}

\noindent
Therefore, we obtain the formula in the statement. Since $\nabla^g\zeta$ is a symmetric tensor:
\begin{equation*}
\dd\zeta(X,Y)=(\nabla_X^g\zeta)(Y)-(\nabla_Y^g\zeta)(X)=0,
\end{equation*}

\noindent
thus $\zeta$ is closed and $\mathcal{D}=\ker(\zeta)$ is an integrable distribution by the Frobenius theorem.
\end{proof}

\begin{remark}\label{rmk:mu_exact}
In the proof of Lemma~\ref{lemma:nabla_zeta}, we have shown that $X(\escal{\rho,\overline{\rho}})=2\mu\escal{\rho,\overline{\rho}}\zeta(X)$ for all $X\in\Gamma(TM)$ or, equivalently, $\dd (\log\langle\rho , \overline{\rho} \rangle) =2\mu\zeta$. Hence, $\mu\zeta$ is an exact one-form.
\end{remark}

\begin{lemma}\label{lemma:nabla_omega}
The fundamental two-form $\omega\in\Omega^2(M)$ satisfies $\nabla^g_X\omega = \mu(\iota_X\omega \wedge \zeta)$ for all $X \in \Gamma(TM)$.
\end{lemma}

\begin{proof}
From Lemma~\ref{lemma:nabla_zeta} we get $\nabla^g_X\zeta=\mu(X^\flat - \zeta(X)\zeta)$, thus $\nabla^g_X\xi=(\nabla^g_X\zeta)^\sharp=\mu(X - \zeta(X)\xi)$. Using $\iota_{\xi}\omega=0$ we compute the covariant derivative evaluated with the Reeb vector field:
\begin{align*}
(\nabla^g_X\omega)(Y,\xi)&=-\omega(Y,\nabla^g_X\xi)=-\omega(Y, \mu(X-\zeta(X)\xi))\\
&=-\mu\omega(Y,X)=\mu\omega(X,Y).
\end{align*}

\noindent
On the other hand, evaluating the target expression yields:
\begin{equation*}
\mu(\iota_X\omega \wedge \zeta)(Y,\xi) = \mu\big(\omega(X,Y)\zeta(\xi) - \zeta(Y)\omega(X,\xi)\big) = \mu\omega(X,Y).
\end{equation*}

\noindent
Hence, the two expressions agree when one of the arguments is $\xi$.\medskip

It remains to show that $(\nabla^g_X\omega)(Y,Z)=0$ for all $Y,Z\in\mathcal{D}=\ker(\zeta)$. We consider separately the cases $X=\xi$ and $X \in \mathcal{D}$.\medskip

First, let us consider $X=\xi$. From Proposition~\ref{prop:KillingSpinorBilinear} we have $\nabla^{g,A}_\xi\rho=i^n \mu *(\zeta\wedge\rho)$. We know that $\rho \wedge \overline{\rho} = (-i)^n \escal{\rho,\overline{\rho}}*\zeta$, which implies $\zeta \wedge \rho \wedge \overline{\rho} = (-i)^n \escal{\rho,\overline{\rho}} \nu$. Therefore, $*\overline{\rho} = i^n \overline{\rho} \wedge \zeta$, and by complex conjugation, $*\rho = (-i)^n \rho \wedge \zeta$. Using $*(X^\flat \wedge \rho) = (-1)^n \iota_X (*\rho)$, we obtain:
\begin{equation*}
\nabla^{g,A}_\xi\rho = i^n\mu *(\zeta\wedge\rho) = i^n \mu (-1)^n \iota_\xi(*\rho) = i^n \mu (-1)^n \iota_\xi((-i)^n \rho \wedge \zeta) = \mu \rho.
\end{equation*}

\noindent
Consequently, $\nabla^{g,A}_\xi\overline{\rho}=\mu\overline{\rho}$. Using $\nabla^{g}_\xi(\rho\triangle_{n-1}\overline{\rho})=(\nabla^{g,A}_\xi\rho)\triangle_{n-1}\overline{\rho}+\rho\triangle_{n-1}(\nabla^{g,A}_\xi\overline{\rho})$ we conclude that $\nabla^{g}_\xi(\rho\triangle_{n-1}\overline{\rho})=2\mu(\rho\triangle_{n-1}\overline{\rho})$. Since $\xi(\escal{\rho,\overline{\rho}})=2\mu\escal{\rho,\overline{\rho}}$ from the proof of Lemma~\ref{lemma:nabla_zeta}, we deduce that:
\begin{equation*}
\nabla^g_\xi\omega = \nabla^g_\xi \left( \frac{-i}{\escal{\rho,\overline{\rho}}} \rho\triangle_{n-1}\overline{\rho} \right) = \frac{-i}{\escal{\rho,\overline{\rho}}} 2\mu (\rho\triangle_{n-1}\overline{\rho}) - \frac{-i}{\escal{\rho,\overline{\rho}}^2} (2\mu\escal{\rho,\overline{\rho}}) \rho\triangle_{n-1}\overline{\rho} = 0.
\end{equation*}

\noindent
Now assume that $X$, $Y$, $Z\in\mathcal{D}$. Then:
\begin{equation*}
*(X^\flat\wedge\rho)=(-1)^n\iota_X(*\rho)=(-1)^n\iota_X((-i)^n\rho\wedge\zeta) = i^n \iota_X\rho\wedge\zeta
\end{equation*}

\noindent
and $\nabla^{g,A}_X\rho=i^n\mu*(X^\flat\wedge\rho)= i^{2n} \mu\iota_X\rho\wedge\zeta = (-1)^n \mu \iota_X\rho\wedge\zeta$. Since $\iota_{\xi}\rho=\iota_{\xi}\overline{\rho}=0$ we have:
\begin{align*}
\escal{\iota_Y(\iota_X\rho\wedge\zeta),\iota_Z\overline{\rho}}&=(-1)^{n-2}\escal{\zeta\wedge\iota_Y\iota_X\rho,\iota_Z\overline{\rho}}\\
&=(-1)^{n-2}\escal{\iota_Y\iota_X\rho,\iota_\xi\iota_Z\overline{\rho}}=0.
\end{align*}

\noindent
Since $X\in\mathcal{D}$, we have $X(\escal{\rho,\overline{\rho}})=0$. A straightforward computation using $\omega=(i\escal{\rho,\overline{\rho}})^{-1}\rho\triangle_{n-1}\overline{\rho}$, $(\rho\triangle_{n-1}\overline{\rho})(Y,Z)=\escal{\iota_Y\rho,\iota_Z\overline{\rho}}-\escal{\iota_Y\overline{\rho},\iota_Z\rho}$, and the above result concludes that $(\nabla^g_X\omega)(Y,Z)=0$ for all $X$, $Y$, $Z\in\mathcal{D}$.
\end{proof}

The formula in Lemma~\ref{lemma:nabla_omega} is equivalent to $(\nabla^g_X\varphi)Y=\mu\big(g(\varphi(X),Y)\xi-\zeta(Y)\varphi(X)\big)$, which is precisely the integrability condition for the tuple $(\varphi,\xi,\zeta)$ to constitute a well-defined $\mu$-Kenmotsu structure on $(M,g)$.\medskip

Therefore, we have proved the following result.

\begin{prop}
Let $(M,g)$ be a $(2n+1)$-dimensional oriented and spin-c Riemannian manifold admitting a pure spin-c Killing spinor with imaginary Killing function $i\mu$, where $\mu \in C^\infty(M, \mathbb{R})$. Then $(M,g)$ is an exact $\mu$-Kenmotsu manifold.
\end{prop}

We conclude this section with a seven-dimensional example illustrating the role of the purity assumption in Theorem~\ref{thm:maintheorem}. More precisely, there exist seven-dimensional Riemannian spin manifolds carrying imaginary Killing spinors which are not pure and whose induced almost contact metric structure is not Kenmotsu. To conclude that, we first need the following lemma.

\begin{lemma}\label{lemma:pure-implies-type-I}
Every pure irreducible complex spinor $\eta \in \Gamma (S)$ on $(M,g)$ satisfies $g(V_\eta,V_\eta) = \scS(\eta,\eta)^2$, where $V_\eta\in\Gamma(TM)$ is the Dirac current defined by $g(V_\eta, X) := \scS(X \cdot \eta, \eta)$ and $\scS$ is a positive-definite admissible Hermitian pairing of positive adjoint type.
\end{lemma}

\begin{proof}
Let $\xi$ be a unit vector field orthogonal to $\mathrm{Ann}(\eta)\oplus\overline{\mathrm{Ann}(\eta)}$. Since $\xi\cdot\eta$ has the same annihilator as $\eta$, the purity condition implies $\xi\cdot\eta=\varepsilon \eta$ for some $\varepsilon\in\{\pm1\}$. If $g(X,\xi)=0$, then $\eta$ and $X\cdot\eta$ belong to different eigenspaces of $(\xi\cdot)$, and hence $\scS(X\cdot\eta,\eta)=0$. Therefore, $V_\eta=\varepsilon\scS(\eta,\eta)\xi$, and, consequently, $g(V_\eta,V_\eta) = \scS(\eta,\eta)^2$.
\end{proof}

\begin{example}
\label{ex:non-pure-example}
Let $\mu>0$ and let $(F^6,h,J)$ be a strictly nearly Kähler spin manifold which is not locally isometric to the round sphere, normalized so that it carries real Killing spinors with Killing numbers $\pm\mu/2$. Then the warped product:
\begin{equation*}
M=(0,+\infty)\times F,  \qquad g=\mathrm{d}t^2+\sinh^2(\mu t)h
\end{equation*}

\noindent
admits an imaginary Killing spinor $\eta$ with Killing function $i\mu$, see~\cite[Ex.~3.2]{C2026}. However, this spinor satisfies $g(V_\eta,V_\eta)<\scS(\eta,\eta)^2$. Thus, $\eta$ is not pure, since every pure irreducible complex spinor in odd Riemannian dimension satisfies, by Lemma~\ref{lemma:pure-implies-type-I}, $g(V_\eta,V_\eta) = \scS(\eta,\eta)^2$. Moreover, we claim that the almost contact metric structure defined by:
\begin{equation*}
\xi= \frac{\partial}{\partial t}, \qquad \zeta=\mathrm{d}t, \qquad \varphi|_{TF}=J,\qquad \varphi(\xi)=0
\end{equation*}

\noindent
is not Kenmotsu. In fact, using the formulas for warped product manifolds, we obtain that:
\begin{equation*}
(\nabla_X^g\varphi)Y = (\nabla_X^hJ)Y + \mu\coth(\mu t)\,g(JX,Y)\xi
\end{equation*}

\noindent
for $X$, $Y\in\Gamma(TF)$. Since $(F,h,J)$ is strictly nearly Kähler, we have that $\nabla^hJ\neq0$. Therefore, Equation~\eqref{eq:Kenmotsu} cannot hold.
\end{example}


\section{The \emph{if} direction}
\label{sec:if}


Assume that $(M,g)$ is a $(2n+1)$-dimensional exact $\mu$-Kenmotsu manifold, where $\mu \in C^\infty(M, \mathbb{R})$. By definition, $(M,g)$ admits an almost contact metric structure $(g,\varphi,\xi,\zeta)$, where $\xi$ is the Reeb vector field, $\zeta = \xi^\flat$ is the almost contact one-form, and $\varphi$ is the fundamental $(1,1)$-tensor field acting as an isometry on the integrable distribution $\mathcal{D} := \ker(\zeta)$ and satisfying $\varphi^2 = -\mathrm{Id} + \zeta \otimes \xi$. The exact $\mu$-Kenmotsu condition amounts to $\mu\zeta\in \Omega^1(M)$ being exact, together with the following equation:
\begin{equation*}
(\nabla^g_X\varphi)Y=\mu\big(g(\varphi(X),Y)\xi-\zeta(Y)\varphi(X)\big),
\end{equation*}

\noindent
where $\nabla^g$ is the Levi-Civita connection. The restriction:
\begin{equation*}
J = -\varphi\vert_{\mathcal{D}}
\end{equation*}

\noindent
defines an almost complex structure on the real rank-$2n$ distribution $\mathcal{D} \subset TM$, which induces a splitting of its complexification into holomorphic and anti-holomorphic subbundles:
\begin{equation*}
\mathcal{D}_{\mathbb{C}} := \mathcal{D}\otimes \mathbb{C} = \mathcal{D}^{(1,0)} \oplus \mathcal{D}^{(0,1)}.
\end{equation*}

\noindent
Since $\mathcal{D}^{(1,0)}$ is a complex vector bundle of rank $n$, its top exterior power defines a complex line bundle over $M$. We identify the characteristic line bundle $\mathcal{L}$ of the underlying $\mathrm{Spin}^c(2n+1)$-structure with this determinant bundle, setting $\mathcal{L} \cong \wedge^n_{\mathbb{C}} \mathcal{D}^{(1,0)}$.

\begin{remark}
\label{remark:canonical-spinc}
Every $\mu$-Kenmotsu manifold $M^{2n+1}$ is, in particular, an almost contact metric manifold. Hence $TM$ admits a reduction of the structure group from $\mathrm{SO}(2n+1)$ to $\mathrm{U}(n)$. Since the standard inclusion $\mathrm{U}(n)\subset\mathrm{SO}(2n)\subset \mathrm{SO}(2n+1)$ admits a canonical lift $\mathrm{U}(n)\to\mathrm{Spin}^{c}(2n+1)$, see e.g.~\cite{Spin89}, this reduction induces a canonical spin-c structure on $M$.
\end{remark}

\noindent
Let $\{e_1,\ldots,e_{2n},e_{2n+1}=\xi\}$ be a local adapted orthonormal frame on $(M,g)$ such that $J(e_{2j-1}) = e_{2j}$ and $J(e_{2j}) = -e_{2j-1}$ for $j=1, \ldots, n$, with corresponding dual coframe $\{e^1,\ldots,e^{2n},e^{2n+1}=\zeta\}$. We consider $(M,g)$ to be equipped with the orientation determined by the volume form:
\begin{equation*}
\nu = (-1)^{\binom{n}{2}} \frac{(-1)^n}{n!} \omega^n \wedge \zeta,
\end{equation*}

where $\omega=g(\varphi-,-)$ is the fundamental two-form of the almost contact metric structure. Under our chosen local adapted orthonormal frame, the Riemannian volume form is thus given by:
\begin{equation*}
\nu = (-1)^{\binom{n}{2}} e^1 \wedge \cdots \wedge e^{2n} \wedge \zeta.
\end{equation*}

Under this choice of frame, a local frame for $\mathcal{D}^{(1,0)}$ is given by the complex vector fields:
\begin{equation*}
Z_j := \frac{1}{2}(e_{2j-1} - ie_{2j}), \qquad j=1, \ldots, n,
\end{equation*}

\noindent
and a local frame for the dual bundle $\mathcal{D}^{(1,0)\ast}$ is given by the complex one-forms:
\begin{equation*}
\theta^j := e^{2j-1} + ie^{2j}, \qquad j=1, \ldots, n.
\end{equation*}

\noindent
Since $\mu\zeta\in \Omega^1(M)$ is exact, we may choose a function $f \in C^\infty(M, \mathbb{R})$ such that $\dd f = \mu\zeta$. We define the $\mathcal{L}$-valued complex exterior $n$-form $\rho \in \Omega^n_{\mathbb{C}}(M, \mathcal{L})$ as follows:
\begin{equation*}
\rho := e^f \theta \otimes \mathfrak{l} := e^f(e^1 + i e^2) \wedge \cdots \wedge (e^{2n-1} + i e^{2n}) \otimes \mathfrak{l},
\end{equation*}

\noindent
where, using the bundle isomorphism $\mathcal{L} \cong \wedge^n_{\mathbb{C}} \mathcal{D}^{(1,0)}$, the local non-vanishing section $\mathfrak{l} \in \Gamma(\mathcal{L})$ is explicitly identified in terms of the underlying real local orthonormal frame of $(M,g)$ via the complexified exterior product as:
\begin{equation*}
\mathfrak{l} = (e_1 - i e_2) \wedge \cdots \wedge (e_{2n-1} - i e_{2n}) = 2^n Z_1 \wedge \cdots \wedge Z_n \in \Gamma(\wedge^n_{\mathbb{C}} \mathcal{D}^{(1,0)}).
\end{equation*}

\begin{lemma}
\label{lemma:rho-well-defined}
The $\mathcal{L}$-valued complex exterior form $\rho\in\Omega^n_\C(M,\mathcal{L})$ is a globally well-defined Cartan $n$-form.
\end{lemma}

\begin{proof}
Under a local change of the frame of $\mathcal{D}$ by a transition matrix $U \in \mathrm{U}(n)$ such that:
\begin{equation*}
\tilde{Z}_j = \sum_k U_{jk} Z_k    
\end{equation*}

\noindent
the top-degree vector field section transforms as $\tilde{\mathfrak{l}} = \det(U) \mathfrak{l}$. By duality and unitarity, the corresponding coframe forms transform under the complex-conjugate matrix, yielding $\tilde{\theta} = \det(\overline{U})\theta = \overline{\det(U)}\theta$. Their variations cancel out identically under the tensor product since $\overline{\det(U)}\det(U) = 1$. Because the conformal factor $e^f$ is a globally defined scalar function invariant under frame changes, $\rho$ is therefore a globally well-defined, locally decomposable nowhere vanishing $n$-form on $M$ taking values in $\mathcal{L}$. Finally, for every $j \in \{ 1, \dots, n\}$, the one-forms $\theta^j$ are isotropic and mutually orthogonal, so $\rho$ is a Cartan $n$-form.
\end{proof}

\begin{remark}
\label{remark:induced-structure}
A direct computation in the adapted frame shows that the almost contact metric structure induced by $\rho$, as in Lemma~\ref{lemma:varphi_square}, is the one we started from.
\end{remark}

\begin{lemma}
    \label{lem:identidad-algebraica}
    Let $\rho\in\Omega^n_\C(M,\mathcal{L})$ be a Cartan $n$-form and $(\varphi,\xi,\zeta)$ be its associated almost contact metric structure. Then, for every $X \in \Gamma (TM)$, we have:
    \begin{equation*}
        i^n \mu *(X^{\flat} \wedge \rho) = \mu\zeta(X)\rho - \mu(\zeta \wedge \iota_{X}\rho).
    \end{equation*}
\end{lemma}

\begin{proof}
\noindent
Decomposing the vector field $X$ into its horizontal and Reeb components, $X = X_{\mathcal{D}} + \zeta(X)\xi$, we have the dual decomposition $X^{\flat} = X_{\mathcal{D}}^{\flat} + \zeta(X)\zeta$. Applying the Hodge star operator linearly, we get:
\begin{equation*}
*(X^{\flat} \wedge \rho) = *(X_{\mathcal{D}}^{\flat} \wedge \rho) + \zeta(X)*(\zeta \wedge \rho).
\end{equation*}

\noindent
With our choice of orientation $\nu$, we have that $i^n\rho\wedge\overline{\rho} = \escal{\rho,\overline{\rho}}*\zeta$. This implies $\rho \wedge \overline{\rho} = (-i)^n \escal{\rho,\overline{\rho}} *\zeta$, and therefore $\zeta \wedge \rho \wedge \overline{\rho} = (-i)^n \escal{\rho,\overline{\rho}} \nu$. By the complex-bilinear extension of the metric, any complex $n$-form $\beta$ satisfies $\beta \wedge *\overline{\rho} = \escal{\beta,\overline{\rho}}\nu$. Comparing this with $\rho \wedge \big(i^n\overline{\rho} \wedge \zeta\big) = \escal{\rho,\overline{\rho}} \nu$, we deduce that $*\overline{\rho} = i^n  \overline{\rho} \wedge \zeta$. Taking the complex conjugate, we obtain the identity $*\rho = (-i)^n \rho \wedge \zeta$.\medskip

\noindent
For the vertical component, we use the identity $*(\zeta \wedge \rho) = (-1)^n \iota_\xi(*\rho)$:
\begin{equation*}
*(\zeta \wedge \rho) = (-1)^n \iota_\xi((-i)^n \rho \wedge \zeta) = (-1)^n (-i)^n (-1)^n \rho = (-i)^n \rho.
\end{equation*}

\noindent
Multiplying by $i^n \mu \zeta(X)$, we obtain exactly:
\begin{equation*}
i^n \mu \zeta(X)*(\zeta\wedge\rho) = i^n(-i)^n\mu\zeta(X)\rho = \mu\zeta(X)\rho.
\end{equation*}

\noindent
For the horizontal component, since $X_{\mathcal{D}} \in \ker(\zeta)$, we use $*(X_{\mathcal{D}}^\flat \wedge \rho) = (-1)^n \iota_{X_{\mathcal{D}}}(*\rho)$:
\begin{equation*}
*(X_{\mathcal{D}}^{\flat} \wedge \rho) = (-1)^n \iota_{X_{\mathcal{D}}}((-i)^n \rho \wedge \zeta) = (-1)^n (-i)^n (\iota_{X_{\mathcal{D}}} \rho) \wedge \zeta.
\end{equation*}

\noindent
Using $(\iota_{X_{\mathcal{D}}} \rho) \wedge \zeta = (-1)^{n-1} \zeta \wedge \iota_{X_{\mathcal{D}}} \rho = -(-1)^n \zeta \wedge \iota_X \rho$, this simplifies to:
\begin{equation*}
*(X_{\mathcal{D}}^{\flat} \wedge \rho) = -(-1)^{2n}(-i)^n \zeta \wedge \iota_X \rho = -(-i)^n \zeta \wedge \iota_X \rho.
\end{equation*}

\noindent
Multiplying by $i^n \mu$, the horizontal term becomes exactly:
\begin{equation*}
i^n \mu *(X_{\mathcal{D}}^{\flat} \wedge \rho) = - i^n (-i)^n \mu (\zeta \wedge \iota_X \rho) = - \mu(\zeta\wedge\iota_X\rho).
\end{equation*}

\noindent
Combining the horizontal and vertical terms, we obtain exactly:
\begin{equation*}
i^n \mu *(X^{\flat} \wedge \rho) = \mu\zeta(X)\rho - \mu(\zeta \wedge \iota_{X}\rho),
\end{equation*}
concluding the proof.
\end{proof}

\noindent
Up to this point, apart from the underlying almost contact metric structure, we have used only the exactness of $\mu\zeta$. We now make use of the $\mu$-Kenmotsu structure equations to obtain the covariant derivative of $\rho$.\medskip

\noindent
Let $\mathcal{P}_{\mathcal{D}} \colon TM \to \mathcal{D}$ be the orthogonal projection onto the integrable distribution $\mathcal{D} = \ker(\zeta)$, defined explicitly by $\mathcal{P}_{\mathcal{D}}(X) := X - \zeta(X)\xi$ for every $X\in TM$. We define a connection $\nabla^{\mathcal{D}}$ on the real vector bundle $\mathcal{D} \to M$ by projecting the Levi-Civita connection $\nabla^g$:
\begin{equation*}
\nabla^{\mathcal{D}}_X Y := \mathcal{P}_{\mathcal{D}}(\nabla^g_X Y)
\end{equation*}

\noindent
for every $Y \in \Gamma(\mathcal{D})$ and $X\in \Gamma(TM)$. Since $(M,g)$ is a $\mu$-Kenmotsu manifold, the restricted endomorphism $J = -\varphi\vert_{\mathcal{D}}$ acts as a parallel complex structure with respect to $\nabla^{\mathcal{D}}$, meaning $\nabla^{\mathcal{D}}_X (JY) = J \nabla^{\mathcal{D}}_X Y$. Extending $\nabla^{\mathcal{D}}$ complex-linearly to the complexified bundle $\mathcal{D}_{\mathbb{C}}$, it preserves the splitting: 
\begin{equation*}
\mathcal{D}_{\mathbb{C}} = \mathcal{D}^{(1,0)} \oplus \mathcal{D}^{(0,1)}
\end{equation*}

\noindent
and restricts to a canonical connection on the subbundle $\mathcal{D}^{(1,0)}$. This connection naturally induces a connection $\nabla^{\mathrm{Det}}$ on the complex determinant line bundle $\mathcal{L} \cong \wedge^n_{\mathbb{C}} \mathcal{D}^{(1,0)}$, which acts on the local section $\mathfrak{l}$ via the trace of the horizontal connection forms:
\begin{equation*}
\nabla^{\mathrm{Det}}_X \mathfrak{l} = i \sum_{j=1}^n \varpi_{2j-1,2j}(X) \mathfrak{l},
\end{equation*}

\noindent
where $\varpi_{ab}(X):=g(\nabla^g_Xe_a,e_b)$. Therefore, the Hermitian connection $\nabla^A$ on $\mathcal{L}$ that we shall use to define the twisted connection $\nabla^{g,A}$ is defined strictly as the undeformed induced connection:
\begin{equation*}
\nabla^A_X := \nabla^{\mathrm{Det}}_X.
\end{equation*}

\noindent
We now compute the total covariant derivative of $\rho \in\Omega^n_\C(M,\mathcal{L})$ with respect to $\nabla^{g,A} = \nabla^g \otimes \nabla^A$. Locally we have $\rho = e^f \theta\otimes\mathfrak{l}$ with $\theta\in \Omega^n_{\C}(M)$, and by the Leibniz rule, we obtain:
\begin{equation*}
\nabla^{g,A}_X \rho = \dd(e^f)(X) \theta \otimes \mathfrak{l} + e^f (\nabla^g_X \theta) \otimes \mathfrak{l} + e^f \theta \otimes (\nabla^A_X \mathfrak{l}).
\end{equation*}

The $\mu$-Kenmotsu structure equations imply that the Levi-Civita connection mixed components satisfy:
\begin{equation*}
\varpi_{a, 2n+1}(X) = g(\nabla^g_X e_a, \xi) = -g(e_a, \nabla_X^g \xi) = -\mu g(e_a, X - \zeta(X)\xi) = -\mu e^a(X).
\end{equation*}

\noindent
For the complex one-forms $\theta^j$, the vertical projection yields the explicit relation:
\begin{equation*}
\nabla^g_X \theta^j = \mathcal{P}_{\mathcal{D}}(\nabla^g_X \theta^j)  - \mu \theta^j(X) \zeta.
\end{equation*}

\noindent
Extending this action to the full wedge product $\theta = \theta^1 \wedge \cdots \wedge \theta^n$, moving $\zeta$ to the front of the wedge product sequentially picks up alternating signs, so the vertical components cleanly add to yield:
\begin{equation*}
(\nabla^g_X \theta)_{\text{vert}} = -\mu \sum_{j=1}^n (-1)^{j-1} \theta^j(X)\zeta \wedge \theta^1 \wedge \cdots \wedge \widehat{\theta^j} \wedge \cdots \wedge \theta^n = -\mu \zeta \wedge \iota_X \theta,
\end{equation*}

\noindent
where the \emph{hat} indicates the omitted factor. Since $\mathfrak{l}=2^n Z_1 \wedge \cdots \wedge Z_n$ with $\{Z_j\}_{j=1}^n$ the $(1,0)$-vectors dual to the local coframe $\{\theta^j\}_{j=1}^n$, its covariant derivative only produces the horizontal trace contribution opposite to that of $\theta$. Evaluating $\nabla^A$ on the local section $\mathfrak{l}$ yields:
\begin{equation*}
\nabla^A_X \mathfrak{l} = \Big(i\sum_{j=1}^n \varpi_{2j-1, 2j}(X)\Big)\mathfrak{l}.
\end{equation*}

\noindent
When evaluating the total covariant derivative $\nabla^{g,A}_X \rho$, we substitute $\dd(e^f)(X) = \mu\zeta(X)e^f$. The local horizontal trace connection forms $i\sum_{j=1}^n \varpi_{2j-1, 2j}(X)$ from $\nabla^A_X \mathfrak{l}$ and $-i\sum_{j=1}^n \varpi_{2j-1, 2j}(X)$ from $\nabla^g_X \theta$ cancel out identically. Then we obtain:
\begin{align*}
\nabla^{g,A}_X\rho &= \mu\zeta(X)\rho + e^f(\nabla^g_X\theta)\otimes\mathfrak{l} + e^f\theta\otimes(\nabla^A_X\mathfrak{l})\\
&= \mu\zeta(X)\rho + e^f\Big(-i\sum_{j=1}^n\varpi_{2j-1,2j}(X)\theta - \mu\zeta\wedge\iota_X\theta\Big)\otimes\mathfrak{l} + e^f\theta\otimes\Big(i\sum_{j=1}^n\varpi_{2j-1,2j}(X)\Big)\mathfrak{l}\\
&= \mu\zeta(X)\rho - \mu(\zeta\wedge\iota_X\rho).
\end{align*}

\noindent
Finally, by Lemma~\ref{lem:identidad-algebraica}, the expression for $\nabla_{X}^{g,A}\rho$ in Proposition \ref{prop:KillingSpinorBilinear} matches the expression above, concluding the following result.

\begin{prop}
Let $(M,g)$ be a $(2n+1)$-dimensional exact $\mu$-Kenmotsu manifold. Then $(M,g)$ admits a pure spin-c Killing spinor with imaginary Killing function $i\mu$, where $\mu \in C^\infty(M, \mathbb{R})$.
\end{prop}


\section{Metric connection interpretation}
\label{sec:metric-connection-interpretation}


In this section, we reinterpret Theorem~\ref{thm:maintheorem} in terms of metric connections with vectorial torsion, following~\cite{C2026}. Let $(M,g,\varphi,\xi,\zeta)$ be an almost contact metric manifold and $\mu\in C^\infty(M,\R)$. We consider the metric connection $\hat\nabla$ with vectorial torsion $\mu\zeta$, determined by:
\begin{equation}
\label{eq:vectorial-connection}
    \hat\nabla_XY := \nabla^g_XY - \mu(\zeta(Y)X-g(X,Y)\xi)
\end{equation}

\noindent
for all $X$, $Y \in \Gamma(TM)$. Its difference tensor with the Levi-Civita connection, $S_X := \hat\nabla_X - \nabla^g_X$, satisfies $g(S_XY,Z)=-g(Y,S_XZ)$ and hence takes values in $\mathfrak{so}(TM)$. Therefore, $\hat\nabla g=0$ and $\hat\nabla\nu=0$. Extended as a derivation, $S_X$ acts on every $k$-form $\beta$ by:
\begin{equation}
\label{eq:difference-forms}
    S_X\beta = \mu (\zeta \wedge \iota_X \beta - X^\flat \wedge \iota_\xi \beta).
\end{equation}

\begin{lemma}
\label{lemma:vectorial-Kenmotsu}
An almost contact metric manifold $(M,g,\varphi,\xi,\zeta)$ is $\mu$-Kenmotsu if and only if $\hat\nabla\varphi=0$, where $\hat\nabla$ is the connection in Equation~\eqref{eq:vectorial-connection} associated with $\mu$.
\end{lemma}

\begin{proof}
For $X$, $Y\in \Gamma(TM)$, a direct computation using
\eqref{eq:vectorial-connection} yields:
\begin{equation*}
    (\hat\nabla_X\varphi)Y=(\nabla^g_X\varphi)Y+\mu\big(g(X,\varphi (Y))\xi+\zeta(Y)\varphi(X)\big).
\end{equation*}

\noindent
Therefore $\hat\nabla\varphi = 0$ is equivalent to Equation~\eqref{eq:Kenmotsu}.
\end{proof}

Furthermore, when $\hat\nabla\varphi = 0$ we obtain $\hat\nabla\xi = 0$, since $\hat\nabla_X\xi$ lies in $\ker(\varphi)=\R\xi$ and is orthogonal to the unit vector field $\xi$, and consequently $\hat\nabla\zeta = 0$.\medskip

\noindent
In~\cite[Thm.~5.1]{C2026} it is shown that if $\eta\in\Gamma(S)$ is a spin-c Killing spinor satisfying $g(V_\eta,V_\eta) = \scS (\eta,\eta)^2$, where $V_\eta\in\Gamma(TM)$ is the Dirac current defined by $g(V_\eta, X) := \scS(X \cdot \eta, \eta)$, then the normalized spinor $\eta/\sqrt{\scS(\eta,\eta)}$ is parallel with respect to $\hat\nabla^A := \hat\nabla \otimes \nabla^A$. We now prove the analogous statement for Cartan $n$-forms. We first observe that the normalization $\rho_0:=\escal{\rho,\overline{\rho}}^{-1/2}\rho$ of a Cartan $n$-form $\rho$ is again a Cartan $n$-form. By Section~\ref{sec:onlyif}, $\rho$ and $\rho_0$ induce the same almost contact metric structure $(\varphi, \xi, \zeta)$.

\begin{prop}
\label{prop:killing-vs-parallel}
Let $(M,g)$ be a $(2n+1)$-dimensional oriented Riemannian manifold equipped with a Hermitian complex line bundle $\mathcal{L}$ with Hermitian connection $\nabla^A$, and let $\mu\in C^\infty(M,\R)$. Then the following two statements are equivalent:
\begin{enumerate}
    \item[\normalfont(1)] There exists a Cartan $n$-form $\rho\in\Omega^n_\C(M,\mathcal{L})$ satisfying $\nabla^{g,A}_X\rho=i^n\mu*(X^\flat\wedge\rho)$ for every $X\in\Gamma(TM)$.
    \item[\normalfont(2)] There exists a $\hat\nabla^A$-parallel and unit Cartan $n$-form $\rho_0\in\Omega^n_\C(M,\mathcal{L})$ whose associated one-form $\zeta:=i^n*(\rho_0\wedge\overline{\rho}_0)$ satisfies that $\mu\zeta$ is exact.
\end{enumerate}
\end{prop}

\begin{proof}
Let $\rho\in\Omega^n_\C(M,\mathcal{L})$ be an arbitrary Cartan $n$-form and let $(\varphi,\xi,\zeta)$ be the almost contact metric structure it induces, see Lemma~\ref{lemma:varphi_square}, which together with $\mu$ determines the connection $\hat\nabla$ of Equation~\eqref{eq:vectorial-connection}. Since $\iota_\xi\rho=0$ by Equation~\eqref{eq:contraccion-xi}, we apply Equation~\eqref{eq:difference-forms} to $\rho$ and obtain:
\begin{equation*}
\hat\nabla^A_X\rho=\nabla^{g,A}_X\rho+\mu\,\zeta\wedge\iota_X\rho
\end{equation*}

\noindent
for all $X\in\Gamma(TM)$. Combining this with Lemma~\ref{lem:identidad-algebraica}, we have:
\begin{equation*}
    \nabla^{g,A}_X\rho-i^n\mu*(X^\flat\wedge\rho)=\hat\nabla^A_X\rho-\mu\zeta(X)\rho.
\end{equation*}

\noindent
Hence, $\rho$ satisfies the imaginary Killing equation if and only if $\rho$ satisfies $\hat\nabla^A_X\rho=\mu\zeta(X)\rho$.\medskip

\noindent
Assume the first item and set $u:=\escal{\rho,\overline{\rho}}$. Then, for all $X\in \Gamma(TM)$ we have:
\begin{equation*}
    X(u) = \hat\nabla_X^A (\escal{\rho, \overline{\rho}}) = \escal{\hat\nabla_X^A \rho, \overline{\rho}} + \escal{\rho,\hat\nabla_X^A \overline{\rho}} = 2\mu \zeta(X) u.
\end{equation*}
In particular, $2\mu\zeta = \dd(\log u)$, so $\mu\zeta$ is exact. A direct computation then gives $\hat\nabla^A\rho_0 = 0$ for $\rho_0 := u^{-1/2} \rho$.\medskip

\noindent
Conversely, assume now the second item and set $\rho := e^{f} \rho_0$, where $f$ is a function such that $\dd f = \mu\zeta$. A direct computation gives $\hat\nabla^A_X \rho = X(f) \rho = \mu \zeta(X) \rho$. Equivalently, $\rho$ satisfies $\nabla^{g,A}_X\rho=i^n\mu*(X^\flat\wedge\rho)$.
\end{proof}

\noindent
The metric connection formalism developed in this section yields an alternative proof of Theorem~\ref{thm:maintheorem}. It relies only on the algebraic content of Lemmas~\ref{lemma:varphi_square}, \ref{lemma:rho-well-defined}, and~\ref{lem:identidad-algebraica}. In particular, it avoids the computation of $\nabla^g\zeta$ and $\nabla^g\omega$ (Lemmas~\ref{lemma:nabla_zeta} and~\ref{lemma:nabla_omega}), as well as that of $\nabla^{g,A}\rho$ carried out in Section~\ref{sec:if}.

\begin{proof}[Proof II of Theorem~\ref{thm:maintheorem}]
Assume that $(M,g)$ admits a pure spin-c Killing spinor with imaginary Killing function $i\mu$, and let $\rho\in\Omega^n_\C(M,\mathcal{L})$ be its complex-bilinear square, which satisfies $\nabla^{g,A}_X\rho=i^n\mu*(X^\flat\wedge\rho)$ by Proposition~\ref{prop:KillingSpinorBilinear}. By Proposition~\ref{prop:killing-vs-parallel}, the one-form $\mu\zeta$ is exact and the normalization $\rho_0:=\escal{\rho,\overline{\rho}}^{-1/2}\rho$ is $\hat\nabla^A$-parallel. Furthermore, both $\rho$ and $\rho_0$ induce the same almost contact metric structure $(\varphi,\xi,\zeta)$. Since $\rho_0$ is $\hat\nabla^A$-parallel, $\hat\nabla$ is metric and $\nabla^A$ is Hermitian, the tensors:
\begin{equation*}
    \zeta= i^n *(\rho_0\wedge\overline{\rho}_0),\qquad \omega= -i \rho_0\triangle_{n-1}\overline{\rho}_0,\qquad \varphi(X)=(\iota_X\omega)^\sharp
\end{equation*}

\noindent
are $\hat\nabla$-parallel. Hence $\hat\nabla\zeta=0$ and $\hat\nabla\varphi=0$. Lemma~\ref{lemma:vectorial-Kenmotsu} then shows that $(M,g,\varphi,\xi,\zeta)$ is $\mu$-Kenmotsu, and it is exact by the first paragraph. Note that, by Equation~\eqref{eq:difference-forms}, the difference tensor acts on $\zeta$ by $S_X\zeta=\mu(\zeta(X)\zeta-X^\flat)$, so that $\hat\nabla\zeta=0$ is precisely the identity obtained in Lemma~\ref{lemma:nabla_zeta}.\medskip

\noindent
Conversely, let $(M,g,\varphi,\xi,\zeta)$ be an exact $\mu$-Kenmotsu manifold, equipped with its canonical spin-c structure with characteristic line bundle $\mathcal{L}\cong\wedge^n_{\mathbb{C}}\mathcal{D}^{(1,0)}$, see Remark~\ref{remark:canonical-spinc} and the discussion above it. By Lemma~\ref{lemma:vectorial-Kenmotsu} and Equation~\eqref{eq:vectorial-connection}, $\hat\nabla$ preserves $g$, $\varphi$, and $\zeta$, hence also $\mathcal{D}^{(1,0)}$. Therefore $\hat{\nabla}$ induces a connection $\nabla^A$ on the complex line bundle $\mathcal{L}$.\medskip

\noindent
In the notation of Section~\ref{sec:if}, let $\rho_0:=2^{-n}\theta\otimes\mathfrak l$ which, by Lemma~\ref{lemma:rho-well-defined}, is a globally defined and unit Cartan $n$-form. Moreover, $\rho_0$ induces the original almost contact metric structure $(\varphi,\xi,\zeta)$. Since $\nabla^A$ is induced by $\hat\nabla$, the connection $\hat\nabla^A$ is the one induced by $\hat\nabla$ on $\Omega^n_\C(M,\mathcal{L})$. Let $X\in\Gamma(TM)$. We use that $\mathcal{L}$ is a complex line bundle and we obtain that (locally) there exists a complex one-form $a$ such that $\nabla^A_X \mathfrak l = a(X)\mathfrak l$. Furthermore, using that $2^{-n}\theta (\mathfrak l) = 1$, we have that $\hat\nabla_X\theta = -a(X)\theta$ and, therefore, $\hat\nabla^A \rho_0=0$. Since $\mu\zeta$ is exact, Proposition~\ref{prop:killing-vs-parallel} yields a Cartan $n$-form satisfying the equation of Proposition~\ref{prop:KillingSpinorBilinear}, which therefore corresponds to the required pure spin-c Killing spinor.
\end{proof}


\section{The global structure}
\label{sec:global}


We have proven that pure spin-c Killing spinors with imaginary Killing function $i\mu$, where $\mu \in C^{\infty}(M,\R)$, correspond to exact $\mu$-Kenmotsu structures $(g,\varphi,\xi,\zeta)$. By the properties of $\mu$-Kenmotsu manifolds, the one-form $\mu\zeta$ is closed if and only if $\dd\mu\wedge\zeta=0$. This condition is equivalent to:
\begin{equation*}
\dd\mu = k\zeta
\end{equation*}

\noindent
for a unique function $k\in C^{\infty}(M,\R)$. Evaluating this equation on $\xi$, we obtain $k=\xi(\mu)$. On the other hand, in the previous section, we introduced the function $f\in C^{\infty}(M,\R)$ via the relation $\dd f = \mu\zeta$, which is consistent thanks to the exact condition on $(g,\varphi,\xi,\zeta)$. Hence $\mu = \xi(f)$ and $k=\xi(\xi(f))$.\medskip

\noindent
The function $f$ already gives us a first consequence for exact $\mu$-Kenmotsu manifolds. We can consider the conformal change $\tilde g = e^{-2f} g$ and the rescaled almost contact structure $(\tilde\varphi,\tilde\xi,\tilde\zeta) := (\varphi,e^f\xi,e^{-f}\zeta)$. The Levi-Civita connection $\nabla^{\tilde g}$ of $\tilde g$ is given by:
\begin{equation*}
    \nabla^{\tilde g}_XY = \nabla^g_XY - \mu\zeta(Y)X + \mu g(X,Y)\xi - \mu\zeta(X)Y
\end{equation*}

\noindent
for all $X$, $Y\in\Gamma(TM)$. In particular, $\nabla^{\tilde g} = \hat\nabla - \mu\zeta \otimes \mathrm{Id}$. Thus, since $\hat\nabla\tilde\varphi = 0$ and $\mathrm{Id}$ acts (as a derivation) trivially on endomorphisms, we have that $\nabla^{\tilde g}\tilde\varphi = 0$. Consequently, every exact $\mu$-Kenmotsu manifold is globally conformally co-Kähler, that is, $\nabla^{\tilde g}\tilde\varphi=0$, $\nabla^{\tilde g}\tilde\xi =0$, and $\nabla^{\tilde g}\tilde\zeta=0$. If we drop this exactness hypothesis and we consider only $\mu$-Kenmotsu manifolds with $\mu\zeta$ closed, then we obtain that $(M,g)$ is locally conformally co-Kähler. These manifolds were described in~\cite[Cor.~5.1]{C2026*}. We now analyze exact $\mu$-Kenmotsu manifolds in terms of $f$, its derivatives $\mu$ and $k$, and the set of zeros of $\mu$.\medskip

Since $\dd f(X) = \mu \zeta (X) = 0 $ for every $X\in \Gamma(\cD)$, it follows that $f$ and all its derivatives with respect to $\xi$ are constant along the leaves of the foliation $\cF$ integrating $\cD = \mathrm{Ker}(\zeta)$, which are necessarily totally umbilical. The foliated structure of $(M,g,\varphi,\xi,\zeta)$ determined by the integrable distribution $\cD$ can be refined by exploiting $f\colon M\to \mathbb{R}$, whose critical points correspond with the zeroes of $\mu$. We define:
\begin{equation*}
Z_\mu:=\{p\in M\mid \mu(p)=0\}.
\end{equation*}

If $Z_{\mu}$ is empty, then the leaves of $\cF$ correspond precisely with the level sets of $f \colon M \to \mathbb{R}$, which defines a smooth submersion. If $Z_{\mu}$ is not empty, then it is a union of leaves of $\cF$. Since:
\begin{equation*}
g(\nabla_X^g Y, \xi) = -g(Y, \nabla_X^g \xi)= -g(Y,  \mu (X - \zeta(X)\xi)) = -\mu g(X,Y)
\end{equation*}

for every $X$, $Y\in \Gamma(\cD)$, it follows that the leaves conforming $Z_{\mu}$ are not only totally umbilical but totally geodesic. To understand how the manifold transitions through these critical leaves, we compute the Hessian of $f$ at a point $p \in Z_\mu$. Taking the covariant derivative of $\dd f = \mu\zeta$ at $p\in Z_{\mu}$, we obtain:
\begin{align*}
\mathrm{Hess}(f)_p (X,Y) &= (\nabla_X^g(\mu\zeta))(Y)\vert_p = X(\mu)\zeta(Y)\vert_p + \mu(p)(\nabla_X^g \zeta)(Y) \vert_p \\
&= \dd\mu (X) \zeta(Y)\vert_p  = k(p) \zeta_p (X) \zeta_p (Y).
\end{align*}

Therefore:
\begin{equation*}
\mathrm{Hess}(f)_p = k(p) \, \zeta_p \otimes \zeta_p = \xi(\xi(f))\vert_p\, \zeta_p \otimes \zeta_p.
\end{equation*}

From now on we assume that $\xi(\xi(f))\neq 0$ on $Z_{\mu}$, that is, the \emph{velocity} of $f$ crosses its zeroes with non-zero acceleration, then $f$ is a Morse-Bott function. Using this Morse-Bott function, we can characterize the global structure of $(M,g)$ in terms of \emph{strata} as follows:
\begin{enumerate}
\item \textit{Regular strata.} The open subset $M \setminus Z_\mu$ decomposes into connected components on which $\mu$ is nowhere vanishing. On each connected component $\mathcal{U} \subset M \setminus Z_\mu$, the restriction $f\colon \cU \to \mathbb{R}$ is a smooth submersion without critical points. Since $\nabla_{\xi}^g \xi = \mu(\xi-\zeta(\xi)\xi) = 0$, the integral curves of the Reeb vector field $\xi$ are unit-speed geodesics orthogonal to the leaves of $\mathcal{F}$. Moreover, since $\dd \zeta = 0$, the flow of $\xi$ preserves the distribution $\mathcal{D}$, providing a foliation-preserving local (or global, under completeness assumptions) diffeomorphism $\mathcal{U} \cong I \times F$, where $I \subset \mathbb{R}$ is an open interval and $F$ is a representative leaf. Using the Lie derivative:
\begin{equation*}
\mathcal{L}_\xi(g(X,Y)) = g(\nabla_X^g \xi, Y) + g(X, \nabla_Y^g \xi) = 2\mu g(X,Y)
\end{equation*}

\noindent
for $X$, $Y \in \Gamma(\mathcal{D})$, the metric on $\mathcal{U}$ takes the local warped product form:
\begin{equation*}
g\vert_{\mathcal{U}} = \dd t^2 + e^{2f(t)} h_F,
\end{equation*}

\noindent
where $t$ is the flow parameter along $\xi$ such that $\dd t = \zeta$, $f'(t) = \mu(t)$, and $h_F$ is a $t$-independent Riemannian metric on $F$. By the properties of the $\mu$-Kenmotsu structure, $(F,h_F)$ inherits a canonical transverse Kähler structure defined by $J = -\varphi\vert_{\mathcal{D}}$.

\item \textit{Critical strata.} The connected components of the critical locus $Z_\mu$ are closed, codimension-one, totally geodesic leaves of the foliation $\mathcal{F}$. The integral curves of the Reeb vector field $\xi$ flow transversally through $Z_\mu$. Along a flow line $\gamma(t)$, the function $\mu(t) = f'(t)$ transitions transversally through zero because $\mu'(t) = k(\gamma(t)) \neq 0$. Thus, $f$ attains a local minimum (if $k > 0$) or local maximum (if $k < 0$) in the transverse direction. Each critical leaf $L \subset Z_\mu$ is therefore a cylindrical boundary gluing together two distinct regular strata: one where the leaves of $\mathcal{F}$ are expanding ($\mu > 0$) and one where they are contracting ($\mu < 0$).
\end{enumerate}

\noindent
The manifold $(M,g)$ decomposes into a sequence of expanding and contracting Kähler cylinders, smoothly glued together along the totally geodesic \emph{equators} corresponding to the critical leaves in $Z_\mu$. This provides a complete description of exact $\mu$-Kenmotsu manifolds (and thus, of odd-dimensional manifolds admitting a pure spin-c Killing spinor with an imaginary Killing function) under the assumption of non-zero acceleration of $\mu$ on its zero set. In addition, if $M$ is compact, the smooth function $f\colon M \to \mathbb{R}$ must attain a global maximum and a global minimum. Consequently, $Z_\mu$ cannot be empty and must contain at least two disjoint totally geodesic leaves. This implies that a compact odd-dimensional Riemannian manifold cannot admit an exact $\mu$-Kenmotsu structure such that the zero set of $\mu$ has a single connected component. We finish this section with a family of compact, exact $\mu$-Kenmotsu manifolds.

\begin{example}
Let $(F,h,J)$ be a compact connected Kähler manifold and let $\mathbb{S}^1:=\R/2\pi\Z$. Then the warped product:
\begin{equation*}
(M,g):=(\mathbb{S}^1 \times F, \dd t^2 + e^{2\sin(t)} h)
\end{equation*}

\noindent
is a compact manifold. Furthermore, $(M,g)$ can be equipped with the almost contact metric structure defined by $\xi= \frac{\partial}{\partial t}$, $\zeta=\mathrm{d}t$, $\varphi|_{TF}=J$, and $\varphi(\xi)=0$. By the formulas for the warped product structures, we conclude that $(M,g,\varphi,\xi,\zeta)$ is a compact $\mu$-Kenmotsu manifold, where $\mu(t) = \cos(t)$. In addition, although $\dd t$ is closed and non-exact on $\mathbb{S}^1$, we have that the form $\mu\zeta = \cos (t) \dd t = \dd (\sin (t))$ is exact. Its zero set:
\begin{equation*}
Z_{\mu} = \{t= \pi/2 \} \cup \{ t= 3\pi/2 \}
\end{equation*}

\noindent
has two connected components, showing that the bound above is sharp. Finally, by Theorem~\ref{thm:maintheorem}, $(M,g)$ is a compact Riemannian manifold admitting a pure spin-c Killing spinor with imaginary Killing function $i\mu$.
\end{example}


\bibliographystyle{myamsplain}
\bibliography{biblio}

\end{document}

%% file: Preamble.tex
\usepackage[a4paper,
left=1in,right=1in,top=1.2in,bottom=1.2in,marginparsep=0.3cm,marginparwidth=2cm,%
footskip=.25in]{geometry}

\usepackage[textsize=footnotesize]{todonotes}

\usepackage{bm}
\usepackage{latexsym, amsmath, amstext, amssymb, amsfonts, amscd, bm, array, multirow, amsbsy, mathrsfs, stmaryrd}
\usepackage{amsthm}
\usepackage{t1enc}
\usepackage[mathscr]{eucal}
\usepackage{indentfirst}
\usepackage{pb-diagram}
\usepackage{graphicx}
\usepackage{fancyhdr}
\usepackage{fancybox}
\usepackage{color}
\usepackage{tikz-cd}
\usetikzlibrary{arrows}
\usepackage[all]{xy}
\usepackage{hyperref}
\usepackage{tikz}
\usepackage{xparse}
\hypersetup{colorlinks=false,pdfborderstyle={/S/U/W 0}}
\usetikzlibrary{matrix}
\usepackage{upgreek}
\usepackage[utf8]{inputenc}
\usepackage[T1]{fontenc}
\usepackage{bbm}
\usepackage{enumitem}
\setlist[enumerate]{leftmargin=*,noitemsep}
\setlist[itemize]{leftmargin=*,noitemsep}
\usepackage{float}
\usepackage{soul}

\newcommand{\doi}[1]{\url{https://doi.org/#1}}
\usepackage[noadjust]{cite}

\makeatletter
\def\@defaultbiblabelstyle#1{[#1]}
\makeatother

\usepackage{url}
\theoremstyle{plain}
\newtheorem{thm}{Theorem}[section]

\newtheorem{prop}[thm]{Proposition}

\newtheorem{lemma}[thm]{Lemma}

\theoremstyle{definition}
\newtheorem{definition}[thm]{Definition}

\theoremstyle{remark}
\newtheorem{remark}[thm]{Remark}
\newtheorem{example}[thm]{Example}

\newtheorem*{ack}{Acknowledgements}

\DeclareFontFamily{U}{rsf}{}
\DeclareFontShape{U}{rsf}{m}{n}{<5> <6> rsfs5 <7> <8> <9> rsfs7 <10-> rsfs10}{}
\DeclareMathAlphabet\Scr{U}{rsf}{m}{n}

\def\Z{\mathbb{Z}}
\def\C{\mathbb{C}}
\def\R{\mathbb{R}}

\def\dd{\mathrm{d}}

\def\Id{\mathrm{Id}}

\newcommand{\be}{\begin{equation*}}
\newcommand{\ee}{\end{equation*}}
\newcommand{\ben}{\begin{equation}}
\newcommand{\een}{\end{equation}}
\newcommand{\beqa}{\begin{eqnarray*}}
\newcommand{\eeqa}{\end{eqnarray*}}
\newcommand{\beqan}{\begin{eqnarray}}
\newcommand{\eeqan}{\end{eqnarray}}

\def\scB{\Scr B}
\def\scS{\Scr H}

\def\U{\mathrm{U}}

\def\cD{\mathcal{D}}

\def\cF{\mathcal{F}}

\def\G_2{\mathrm{G_2}}

\def\G{\mathrm{G}}

\def\R{\mathbb{R}}

\def\cL{\mathcal{L}}

\def\dd{\mathrm{d}}

\newcommand{\escal}[1]{\langle#1\rangle}

\def\cU{\mathcal{U}}

\renewcommand{\ker}{\mathrm{Ker}}

\newcolumntype{P}[1]{>{\centering\arraybackslash}p{#1}}

\usepackage[math]{anttor}
\usepackage{orcidlink}

\allowdisplaybreaks